\documentclass[12pt]{amsart}
\usepackage{dsfont,hyperref}
\usepackage[all]{xy}
\usepackage{amssymb,hyperref}
\SelectTips{cm}{}
\usepackage{fullpage}
\allowdisplaybreaks

\newtheorem{theorem}{Theorem}[section]
\newtheorem{lemma}[theorem]{Lemma}
\newtheorem{proposition}[theorem]{Proposition}
\newtheorem{corollary}[theorem]{Corollary}
\theoremstyle{definition}
\newtheorem{definition}[theorem]{Definition}
\newtheorem{example}[theorem]{Example}
\newtheorem{question}[theorem]{Question}
\newtheorem{conjecture}[theorem]{Conjecture}
\newtheorem{remark}[theorem]{Remark}

\newcommand{\End}{\text{End}}
\newcommand{\Hom}{\text{Hom}}

\newcommand{\C}{\mathcal{C}}

\newcommand{\ot}{\otimes}

\newcommand{\one}{\mathds{1}}
\newcommand{\lk}{\ar@{-}}

\numberwithin{equation}{section}
\begin{document} 

\title{On a necessary condition for unitary categorification of fusion rings}

\begin{abstract}
In \cite{LPW} Liu, Palcoux and Wu proved a remarkable necessary condition for a fusion ring to admit a unitary categorification, by constructing invariants of the fusion ring that have to be positive if it is unitarily categorifiable. The main goal of this note is to provide a somewhat more direct proof of this result. 
In the last subsection we discuss integrality properties of the Liu-Palcoux-Wu invariants. 
\end{abstract}

\author{Pavel Etingof}
\address{Department of Mathematics, Massachusetts Institute of Technology,
Cambridge, MA 02139, USA}
\email{etingof@math.mit.edu}

\author{Dmitri Nikshych}
\address{Department of Mathematics and Statistics,
University of New Hampshire,  Durham, NH 03824, USA}
\email{dmitri.nikshych@unh.edu}

\author{Victor Ostrik}
\address{Department of Mathematics,
University of Oregon, Eugene, OR 97403, USA}
\address{Laboratory of Algebraic Geometry,
National Research University Higher School of Economics, Moscow, Russia}
\email{vostrik@uoregon.edu}

\maketitle

\centerline{\bf To the memory of Earl J. Taft} 

\tableofcontents

\section{Introduction}

In \cite{LPW} Liu, Palcoux and Wu proved a remarkable necessary condition for a fusion ring to be unitarily categorifiable, which can be used to disqualify many new fusion rings from having such a categorification. They did so by constructing invariants of the fusion ring that have to be positive if it is the Grothendieck ring of a unitary fusion category. The first goal of this artcile is to provide a somewhat more direct proof of this criterion, using the notion of a formal codegree and the Drinfeld center of the underlying category, which is done in Section 4, after discussing preliminaries in Sections 2,3. Thus Section 4 does not contain any essentially new results, it is just our attempt to give an exposition of the result of \cite{LPW} in a somewhat more algebraic language.\footnote{We note that the paper \cite{LPW} contains a nice diagrammatic proof of this result, see the proof of Proposition 7.7 there.}

Our second goal is to discuss integrality properties of the Liu-Palcoux-Wu invariants and their relation to the generalized Kaplansky 6th conjecture stating that in any spherical fusion category, dimensions of objects divide the dimension of the category. This is implemented in Section 5. 

{\bf Acknowledgements.} We are very grateful to Sebastien Palcoux for useful discussions. P. E.'s work was partially supported by the NSF grant DMS - 1916120. The work of D.~N. was partially supported by the NSF grant DMS-1800198. The work of V. O. was partially supported by the NSF grant DMS-1702251 and by the Russian Academic Excellence Project `5-100'.

\section{Fusion rings}

Let $A$ be a fusion ring with a finite basis $\{b_i\}$ (where $b_0=1$) and let  $\tau: A\to \Bbb Z$ be the trace function given by 
$\tau(\sum_i \lambda_ib_i):=\lambda_0$.
There is an inner product
 $(a,b):=\tau(ab^*)$, for $a,b\in A$, with respect to which the basis $b_i$ is orthonormal. Let $N_{ijm}:=(b_ib_j,b_m)$ be the fusion coefficients
and let $d_j$ be the Frobenius-Perron dimensions of the basis elements $b_j$. Recall that $N_{ijm}=N_{jm^*i^*}=N_{m^*ij^*}=N_{j^*i^*m^*}$ and $d_j=d_{j^*}$. For more details on fusion rings and Frobenius-Perron dimensions we refer the reader to \cite[Chapter 3]{EGNO}.

We now give a somewhat shorter proof of \cite[Proposition 8.1]{LPW}. 
(This proposition will not be used below.)  

\begin{proposition} (\cite{LPW}, Proposition 8.1) 
\begin{enumerate}
\item[(i)] $\sum_m N_{ijm}^2\le {\rm min}(d_i^2,d_j^2)$.
\item[(ii)] $N_{ijm}\le d_id_j/d_m$; hence $N_{ijm}\le d_m(d_i/d_j)^{\frac{2-t}{t}}$ for any $t\ge 1$.
\item[(iii)] $N_{ijm}\le {\rm min}(d_i,d_j,d_m)$.
\item[(iv)] $\sum_m N_{i_1i_2m}N_{i_3i_4m}\le d_{i_p}d_{i_q}$ for any distinct $1\le p,q\le 4$. 
\end{enumerate}
\end{proposition} 
\begin{proof} 
(i) $\sum_m N_{ijm}^2=\sum_m (b_ib_j,b_m)^2=(b_ib_j,b_ib_j)=(b_{i^*}b_ib_j,b_j)\le {\rm FPdim}(b_{i^*}b_i)=d_i^2$.
Similarly, $\sum_m N_{ijm}^2 \leq d_j^2$.

(ii) The first statement follows from the equality $\sum_m N_{ijm}d_m=d_id_j$. Using the symmetry properties of $N_{ijm}$, this implies that $N_{ijm}\le d_md_i/d_j$ and $N_{ijm}\le d_md_j/d_i$. This implies the second statement, since for $t\ge 1$ the number $\frac{2-t}{t}$ varies between $1$ and $-1$. 

(iii) $N_{ijm}$ are matrix elements of the matrix of multiplication by $b_i$ 
in the orthonormal basis $b_j$, so they do not exceed the norm $d_i$ of this matrix. The rest follows from the symmetry properties of $N_{ijm}$. 

(iv) $\sum_m N_{i_1i_2m}N_{i_3i_4m}=\sum_m N_{i_1i_2m}N_{i_3^*mi_4}=(b_{i_3^*}b_{i_1}b_{i_2},b_{i_4})=(b_{i_1}b_{i_2}b_{i_4^*},b_{i_3})$. 
This is a matrix element of each of the following operators on $A$: $L(b_{i_3^*})L(b_{i_1})$; $R(b_{i_1})R(b_{i_2})$; $L(b_{i_3^*})R(b_{i_2})$; 
$L(b_{i_1})L(b_{i_2})$; $R(b_{i_2})R(b_{i_4^*})$; $L(b_{i_1})R(b_{i_4^*})$, where $L$ denotes left 
multiplication and $R$ right multiplication. 
Thus it is at most the norms of these operators, as desired. 
\end{proof} 

\section{Formal codegrees and the Drinfeld center} 
Recall that $A_{\Bbb C} =A \ot_\mathbb{Z} \mathbb{C}$ is a $*$-algebra. 
Consider an irreducible representation \linebreak $(V,\rho: A_{\Bbb C}\to \End V)$ of $A_{\Bbb C}$ (which we will often shortly denote just by $\rho$). The {\bf formal codegree} $\alpha_\rho$ of $\rho$ is the eigenvalue of the central element $z_\rho:=\sum_i {\rm Tr}(\rho(b_i))b_i^*$ on $\rho$, an algebraic integer (the eigenvalues of this element on all the other irreducible representations are $0$), see e.g. \cite[Lemma 2.3]{O}.  
Thus
$\tau(z_\rho)=\dim \rho=\alpha_\rho^{-1}{\rm Tr} \rho(z_\rho)$. Hence 
$$
\tau(a)=\sum_\rho \alpha_\rho^{-1} {\rm Tr}\rho(a),\ a\in A. 
$$
In particular, $\alpha_\rho>0$ for all $\rho$, see also \cite[Remark 2.12]{O2}.

Let $v_1,v_2\in V$, $f_1,f_2\in V^*$. 

\begin{lemma}\label{matrel} We have 
$$
\sum_i f_1(\rho(b_i)v_1)f_2(\rho(b_i^*)v_2)=\alpha_\rho f_2(v_1)f_1(v_2).
$$
\end{lemma} 

\begin{proof} It suffices to show that for any $a_1,a_2\in A$ we have 
$$
\sum_i {\rm Tr}\rho(a_1b_i){\rm Tr}\rho(b_i^*a_2)=\alpha_\rho {\rm Tr}\rho(a_1a_2); 
$$
then we can take $a_1,a_2$ such that $\rho(a_1)=v_1\otimes f_1^*$, $\rho(a_2)=v_2\otimes f_2^*$ in $\End V$, which gives the desired statement. 
But this is equivalent to the obvious relation
$$
\sum_i {\rm Tr}\rho(a_1b_i)(a_2,b_i)={\rm Tr}\rho(a_1a_2), \qquad a_1,a_2\in A.
$$
\end{proof} 


Suppose now that $\C$ is a spherical fusion category that categorifies  a fusion ring $A$. 
Recall that irreducible (unitary) representations $V$ of $A_{\Bbb C}$ have the form 
$V_Z=\Hom_\C(\one,Z)$ where $Z$ is a simple object of the Drinfeld center 
$\mathcal Z(\C)$ of $\C$ whose image in $\C$ (which we, abusing notation, will also denote by $Z$) contains $\one$ (see \cite[Theorem 2.13]{O2} and \cite[Theorem 5.9]{Sh}). Namely, the map $\rho_Z$ is constructed as follows. 
Note that $V_Z^*=\Hom_\C(Z,\one)$ using the composition pairing. Given $X\in \C$, $v\in V_Z$, $f\in V_Z^*$, we define $\rho_Z(X)$ by 
$$
f(\rho_Z(X)v)={\rm Tr}((1_X\otimes f)\circ c_{X,Z}\circ (v\otimes 1_X)).
$$
where $c_{X,Z}: Z\otimes X\to X\otimes Z$ is the central structure  of $Z$. 
Using the identity
$$
c_{Z,X\otimes Y}=(1_X\otimes c_{Z,Y})\circ (c_{Z,X}\otimes 1_Y),
$$ 
it is not hard to show that $\rho_Z$ is a representation. Note that 
$\rho_Z(X^*)=\rho_Z(X)^\dagger$ (the adjoint operator) and $\rho_{Z^*}\cong \overline \rho_Z$ (the same action on the complex conjugate space). 

\begin{theorem}\label{os} {\em (\cite[Theorem 2.13]{O2})} If $\rho=\rho_Z$ then one has $\alpha_\rho=\frac{\dim\C}{\dim Z}$. In particular, if $[Z:\one]>0$ then $\dim Z>0$.  
\end{theorem}

\begin{corollary}\label{matrel1} Under the assumptions of Lemma \ref{matrel} 
$$
\sum_i f_1(\rho(b_i)v_1)f_2(\rho(b_i^*)v_2)=\frac{\dim\C}{\dim Z}f_2(v_1)f_1(v_2).
$$
\end{corollary} 

\section{Positivity results} 
Recall that a {\bf unitary fusion category} is a fusion category with a $*$-structure (\cite{T}; see also \cite{G}, Subsection 2.1 for a full definition). A {\bf unitary categorification} of a fusion ring is a realization of this ring as the Grothendieck ring of a unitary fusion category. 
  
Let $(V_s,\rho_s: A\to {\rm End} V_s)$ be a collection of irreducible (unitary) representations of $A_{\Bbb C}$ and $v_s\in V_s$, $s=1,...,n$. 

\begin{theorem}\label{mai} (\cite{LPW}, Proposition 8.3) (i) If $A$ admits a unitary categorification then we have 
$$
\sum_i \frac{1}{d_i^{n-2}}(\rho_1(b_i)v_1,v_1)...(\rho_n(b_i)v_n,v_n)\ge 0. 
$$

(ii) (\cite{LPW}, Corollary 8.5) If in (i) $A$ is commutative then 
$$
\sum_i \frac{1}{d_i^{n-2}}\rho_1(b_i)...\rho_n(b_i)\ge 0. 
$$
\end{theorem} 

\begin{example}\label{exa} 1. Let $n=1$. Then Theorem \ref{mai} says that $(v,\rho(\sum_i d_ib_i)v)\ge 0$. This holds (regardless of $A$ being unitarily categorifiable) since $R:=\sum_i d_ib_i$ is the regular element of $A$, hence $Ra={\rm FPdim}(a)R$ for any $a\in A$. 

2. Let $n=2$. Then Theorem \ref{mai} says that $\sum_i (\rho_1(b_i)v_1,v_1)(\rho_2(b_i)v_2,v_2)\ge 0$. This follows (again regardless  of $A$ being unitarily categorifiable) 
from Lemma \ref{matrel}. Indeed, since $\rho(b_i^*)=\rho(b_i)^\dagger$, Lemma \ref{matrel} implies that this sum is zero unless $\rho_1\cong \overline\rho_2$, and 
$$
\sum_i (\rho(b_i)v_1,v_1)(\overline\rho(b_i)v_2,v_2)=\alpha_\rho|(v_1,v_2)|^2.
$$

However, for $n\ge 3$, as shown in \cite{LPW}, unitary categorifiability of $A$ is essential. 
\end{example}

\begin{proof} It suffices to prove (i) for $n\ge 3$. Let $\C$ be a unitary fusion category categorifying $A$. 
Recall (\cite{EGNO}, Section 9.5) that it has a canonical spherical structure in which the dimensions of simple objects $X_i$ are $d_i$; so let us endow $\C$ with this structure.
Let $Z_1,...,Z_n\in \mathcal Z(\C)$, 
and consider the vector space $\Hom_{\mathcal Z(\C)}(\one,Z_1\otimes...\otimes Z_n)$.  
Since $\C$ is unitary, this space has a positive definite Hermitian inner product 
given by $(v,w)=w^\dagger\circ v$.
Now let $Z_1,...,Z_n\in \mathcal Z(\C)$ be objects containing $\one$ as objects of $\C$. Let $(V_i,\rho_i)$ be the corresponding representations of $A_\Bbb C$, $i=1,...,n$. 
We have a natural map 
$$
\phi: V_1\otimes...\otimes V_n\to \Hom_{\mathcal Z(\C)}(\one,Z_1\otimes...\otimes Z_n)
$$ 
given by the orthogonal projection of $v_1\otimes...\otimes v_n\in  \Hom_\C(\one,Z_1\otimes...\otimes Z_n)$ to the space $ \Hom_{\mathcal Z(\C)}(\one,Z_1\otimes...\otimes Z_n)$, where $v_i\in \Hom(\one ,Z_i)$. In other words, we may view $v_1\otimes...\otimes v_n$ as an element of $\Hom_{\mathcal Z(\C)}(Z_1\otimes...\otimes Z_n,\one)^*$ by taking composition, and $\phi(v_1\otimes...\otimes v_n)$ is the corresponding element 
of $\Hom_{\mathcal Z(\C)}(\one, Z_1\otimes...\otimes Z_n)$. This element may be viewed as a $\mathcal Z(\C)$-morphism $Z_1^*\to Z_2\otimes...\otimes Z_n$. 
Now, we have 
\begin{eqnarray*}
\Hom_\C(\one,Z_2\otimes...\otimes Z_n) &=& \oplus_{Z\in {\rm Irr}\mathcal Z(\C)} \Hom_\C(\one,Z^*)\otimes \Hom_{\mathcal Z(\C)}(Z^*,Z_2\otimes...\otimes Z_n) \\
&=& \oplus_{Z\in {\rm Irr}\mathcal Z(\C)} \Hom_\C(Z,\one)\otimes \Hom_{\mathcal Z(\C)}(Z^*,Z_2\otimes...\otimes Z_n).
\end{eqnarray*}
Let $\lbrace{e_{Z,j}\rbrace}$ is an orthonormal basis of $V_Z$. Then we get 
\begin{equation}\label{2...n}
v_2\otimes \cdots \otimes v_n=\sum_{Z\in {\rm Irr}\mathcal Z(\C)} \sum_{j=1}^{\dim V_Z} e_{Z,j}^\dagger\otimes \phi(e_{Z,j}\otimes v_2\otimes...\otimes v_n).
\end{equation}
Recall that 
$$
c_{Z\otimes Z',X}=(c_{Z,X}\otimes 1_{Z'}) \circ (1_{Z}\otimes c_{Z',X}).
$$
Therefore,  using \eqref{2...n}, it is easy to see that 
\begin{eqnarray*}
\lefteqn{\frac{1}{d_i^{n-2}}(\rho_2(X_i)v_2,v_2)...(\rho_n(X_i)v_n,v_n)} \\
&=&
\sum_{Z\in {\rm Irr}\mathcal Z(\C)} \dim Z\sum_{j,k=1}^{\dim V_Z} (\overline\rho_Z(X_i)e_{Z,j},e_{Z,k})(\phi(e_{Z,j}\otimes v_2\otimes...\otimes v_n),\phi(e_{Z,k}\otimes v_2\otimes...\otimes v_n)) \\
&=&
\sum_{Z\in {\rm Irr}\mathcal Z(\C)}\dim Z\sum_{j=1}^{\dim V_Z} (\phi(e_{Z,j}\otimes v_2\otimes...\otimes v_n),\phi(\rho_{Z^*}(X_i)e_{Z,j}\otimes v_2\otimes...\otimes v_n)).
\end{eqnarray*}
Let $B_Z: V_Z\to V_Z$ be an operator whose matrix elements in the basis
$e_{Z,j}$ are 
$$
B_{Z,jk}=(\phi(e_{Z,j}\otimes v_2\otimes...\otimes v_n),\phi(e_{Z,k}\otimes v_2\otimes...\otimes v_n)). 
$$
Then 
$$
\frac{1}{d_i^{n-2}}(\rho_2(X_i)v_2,v_2)...(\rho_n(X_i)v_n,v_n)=
\sum_{Z\in {\rm Irr}\mathcal Z(\C)}\dim Z \sum_{j=1}^{\dim V_Z}(B_Ze_{Z,j},\rho_{Z^*}(X_i)e_{Z,j}). 
$$
Thus,
\begin{eqnarray*}
\lefteqn{\sum_i \frac{1}{d_i^{n-2}}(\rho_1(X_i)v_1,v_1)(\rho_2(X_i)v_2,v_2)...(\rho_n(X_i)v_n,v_n)}\\
&=& \sum_i (\rho_1(X_i)v_1,v_1)\sum_{Z\in {\rm Irr}\mathcal Z(\C)}\dim Z\sum_{j=1}^{\dim V_Z}(B_Ze_{Z,j},\rho_{Z^*}(X_i)e_{Z,j}) \\
&=& \sum_i \dim Z_1(\rho_{Z_1}(X_i)v_1,v_1)\sum_{j=1}^{\dim V_{Z_1}}(B_{Z_1}e_{Z_1,j},\rho_{Z_1^*}(X_i)e_{Z_1,j}) \\
&=&  \sum_{j=1}^{\dim V_{Z_1}}\dim Z_1(B_{Z_1}e_{Z_1,j},v_1)(v_1,\rho_{Z_1^*}(X_i)e_{Z_1,j})\\
&=& (B_{Z_1}v_1,v_1)=\dim \C\cdot ||\phi(v_1\otimes...\otimes v_n)||^2, 
\end{eqnarray*}
where for the last equality we used Lemma \ref{matrel} and Theorem \ref{os}. 
Since $\C$ is unitary, so is its Drinfeld center, so this squared norm is $\ge 0$, as desired. 
\end{proof} 

The formula at the end of the proof of Theorem \ref{mai} can, in fact, be generalized to the situation when the fusion category $\C$ is spherical but not assumed unitary or even Hermitian. Namely, let $v_i\in V_i$, $f_i\in V_i^*$, and 
let 
$$
\psi: V_1^*\otimes...\otimes V_n^*\to \Hom(\one,Z_1\otimes...\otimes Z_n)^*\cong 
\Hom(Z_1\otimes...\otimes Z_n,\one)
$$ 
be the natural map. 

\begin{proposition}\label{p1} We have 
$$
\sum_i \frac{1}{(\dim X_i)^{n-2}}f_1(\rho_1(X_i)v_1)...f_n(\rho_n(X_i)v_n)=
\dim(\C)(\phi(v_1\otimes...\otimes v_n),\psi(f_1\otimes...\otimes f_n)). 
$$
where $\dim$ denotes the categorical dimensions. 
\end{proposition} 

Consider the operator $\psi^*\phi\in \End(V_1\otimes...\otimes V_n)$. 
Proposition \ref{p1} immediately implies 

\begin{corollary} Let $a_1,...,a_n\in A$. Then 
$$
\sum_i \frac{1}{(\dim X_i)^{n-2}}{\rm Tr}\rho_1(a_1X_i)...{\rm Tr}\rho_n(a_nX_i)=
\dim \C\cdot {\rm Tr}\left(\psi^*\phi\circ (\rho_1(a_1)\otimes...\otimes \rho_n(a_n))\right). 
$$
\end{corollary} 

In particular, we have 

\begin{corollary} If $V_i$ are 1-dimensional then
$$
I_n(\rho_1,...,\rho_n):=\sum_i \frac{1}{(\dim X_i)^{n-2}}\rho_1(X_i)...\rho_n(X_i)=
\dim\C\cdot (\phi,\psi).
$$
\end{corollary} 

Here we treat $\phi,\psi$ as vectors since the space $V_1\otimes...\otimes V_n$ is 1-dimensional. 

\begin{remark} It is easy to see that the invariants $I_n$ satisfy the recursion
$$
I_n(\rho_1,...,\rho_n)=\sum_{\rho} \alpha_\rho^{-1}I_{n-1}(\rho_1,...,\rho_{n-2},\rho)I_3(\overline \rho,\rho_{n-1},\rho_n),\ n\ge 3.
$$
(see \cite{LPW}). This implies that the multiplication law on the complexified Grothendieck group of the category of $A_{\Bbb C}$-modules given by 
$$
\rho_1\ast\rho_2=\sum_\rho \alpha_\rho^{-1}I_3(\rho_1,\rho_2,\overline\rho)\rho
$$
is commutative and associative, giving this group the structure of a commutative Frobenius algebra. This is nothing but the dual fusion ring considered in \cite{LPW}. 
\end{remark} 

\section{Integrality properties of spherical fusion categories}

In this section we explore integrality properties for spherical fusion categories over $\Bbb C$. For simplicity we will restrict ourselves to the case of commutative fusion rings. 

\subsection{Isaacs criterion and Frobenius type} 

Let $\C$ be a fusion category with fusion ring $A$. Let $s\ge 0$ be a rational number.   

\begin{definition}\label{pdef} We say that $\C$ is $s$-{\bf Isaacs} if for any character $\rho: A\to \Bbb C$ and any simple object $X\in \C$, the number 
$$
\lambda_s(\rho,X):=(\dim \C)^s(\dim Z_\rho)^{1-s}\frac{\rho(X)}{\dim X}
$$
is an algebraic integer.  
\end{definition} 

Since $\dim Z_\rho$ divides $\dim \C$, if $\C$ is $s$-Isaacs then it is $t$-Isaacs for any $t>s$.  

The $0$-Isaacs property will simply be called the {\bf Isaacs property}. It was introduced in \cite{LPR,LPR2}. This definition was motivated by the following proposition. 

\begin{proposition} Any ribbon fusion category is Isaacs. 
\end{proposition} 

\begin{proof} In a ribbon category $\C$ we have $\rho(X)=\frac{\bold s_{Z_\rho X}}{\dim Z_\rho}$. 
where $(\bold s_{ij})$ is the $S$-matrix of the Drinfeld center $\mathcal Z(\C)$. Thus $$
\lambda_0(\rho,X)=\frac{\bold s_{Z_\rho X}}{\dim X},
$$ 
which is an algebraic integer (an eigenvalue of an integer matrix) by the Verlinde formula, see \cite{EGNO}, Corollary 8.14.4. 
\end{proof} 

In \cite{LPR} Liu, Palcoux and Ren conjectured that any spherical pseudo-unitary fusion category is Isaacs (Conjecture 2.5 in loc. cit.). However, this conjecture was disproved 
in \cite{BP}, where it is shown that the extended Haagerup fusion category is not Isaacs. 
In fact, they show that this category is $s$-Isaacs iff $s\ge 1$. 

Recall that $\C$ is said to be {\bf Frobenius type} if dimensions of its simple objects divide $\dim\C$. The Kaplansky 6th conjecture for fusion categories 
states that any spherical fusion category is Frobenius type. Generalizing, 
let us say that a category $\C$ is $s$-{\bf Frobenius type} for a rational number $s\ge 1$ if the dimensions of its simple objects divide $(\dim \C)^s$; so $1$-Frobenius type is Frobenius type in the usual sense. 
Clearly, $s$-Frobenius type implies $t$-Frobenius type for any 
$t>s$. Finally, it is clear that any Frobenius type category $\C$ is $1$-Isaacs.

This motivates the following weakening of Kaplansky 6-th conjecture for fusion categories. 

\begin{conjecture} Any spherical fusion category is 1-Isaacs. 
\end{conjecture} 

The following proposition  generalizes \cite{LPR}, Proposition 2.6.

\begin{proposition} If $\C$ is $s$-Isaacs for $s\ge \frac{1}{2}$ 
then it is $s+\frac{1}{2}$-Frobenius type. In particular, 
if $\C$ is $\frac{1}{2}$-Isaacs then it is Frobenius type. 
\end{proposition} 

\begin{proof} If $\C$ is $s$-Isaacs with $s\ge \frac{1}{2}$ then the number 
$$
\sum_{\rho} \lambda_s(\rho,X)\lambda_s(\overline{\rho},X)(\dim Z_{\rho})^{2s-1}=
\sum_\rho (\dim \C)^{2s}\dim Z_\rho\frac{\rho(X)\overline\rho(X)}{(\dim X)^2}=\frac{(\dim\C)^{2s+1}}{(\dim X)^2}
$$
is an algebraic integer, hence $\C$ is $s+\frac{1}{2}$-Frobenius type, as claimed. 
\end{proof} 

\subsection{Integrality properties of $I_n$} 

For any characters $\rho_i: A\to \Bbb C$, $i=1,...,n$, and a rational number $s\ge 0$ define the number
$$
J_{n,s}(\rho_1,...,\rho_n):=(\dim\C)^{(n-2)s}(\dim Z_{\rho_1}...\dim Z_{\rho_n})^{1-s}I_n(\rho_1,...,\rho_n). 
$$ 
For example, 
$$
J_{2,s}(\rho,\eta)=\dim\C(\dim Z_\rho)^{1-2s}\delta_{\eta,\overline\rho}.
$$

\begin{theorem} $\C$ is $s$-Isaacs if and only if 
for any $n\ge 2$ and $\rho_i$, $i=1,...,n$, the number $\frac{J_{n,s}(\rho_1,...,\rho_n)}{(\dim Z_{\rho_1}\dim Z_{\rho_2})^{1-s}}$ is an algebraic integer. 
\end{theorem}

\begin{proof}
We have 
$$
\frac{J_{n,s}(\rho_1,...,\rho_n)}{(\dim Z_{\rho_1}\dim Z_{\rho_2})^{1-s}}=
\sum_X \rho_1(X)\rho_2(X)\lambda_{s}(\rho_3,X)...\lambda_{s}(\rho_n,X),
$$
which proves the ``only if" part. For the ``if" part, consider the sum 
$$
\sum_\eta (\dim Z_\eta)^{s}\frac{J_{n,s}(\rho,...,\rho,\eta)}{(\dim Z_\rho)^{2(1-s)}}\overline\eta(Y)=
$$
$$
=(\dim \C)^{(n-2)s+1}(\dim Z_\rho)^{(n-3)(1-s)}\frac{\rho(Y)^{n-1}}{\dim(Y)^{n-2}}=
$$
$$
\dim \C(\dim Z_\rho)^{s-1}\rho(Y)
\lambda_s(\rho,Y)^{n-2}.
$$
This is an algebraic integer for all $n$, hence $\lambda_s(\rho,Y)$ is an algebraic integer, as claimed. 
\end{proof} 

\begin{definition} Let us say that $\C$ is {\bf strongly Isaacs} if the number 
$\frac{J_{n,0}(\rho_1,...,\rho_n)}{\dim\C\sqrt{\dim Z_{\rho_1}\dim Z_{\rho_2}}}$ is an algebraic integer. 
\end{definition} 

It is easy to see that if $\C$ is strongly Isaacs then it is ($0$-)Isaacs. 

\begin{theorem}   
(i) The category $\C={\rm Rep}(G)$ for a finite group $G$ is strongly Isaacs. 

(ii) Any modular category is strongly Isaacs. 
\end{theorem}

\begin{proof} 
(i) The representations $\rho_i$ correspond to the conjugacy classes 
$C_i$ in $G$, and $Z_{\rho_i}={\rm Fun}(C_i)$, so $\dim Z_{\rho_i}=|C_i|$. 
Also it is well known that in this case 
\begin{equation}\label{Jinv}
\frac{J_{n,0}(\rho_1,...,\rho_n)}{\dim\C}=\frac{\dim Z_{\rho_1}...\dim Z_{\rho_n}}{\dim \C}I_n(\rho_1,...,\rho_n)=
|S|,
\end{equation} 
where $S$ is the set of tuples 
$(g_1,...,g_n)\in G^n$ such that $g_1...g_n=1$ and $g_i\in C_i$. 
The set $S$ carries an action of $G$ by conjugation, and the stabilizer 
of every element is contained in the centralizer $G_i$ of the class $C_i$ for all $i$, so 
the size of each orbit is divisible by $|C_i|$. This implies that 
\eqref{Jinv} is an integer and divisible by $\dim Z_{\rho_i}$ for all $i$, hence by $\sqrt{\dim Z_{\rho_1}\dim Z_{\rho_2}}$, which yields the desired statement.

(ii) Each $\rho_j$ corresponds to a simple object $X_j$ of $\C$. A direct calculation using the Verlinde formula yields 
$$
\frac{J_{n,0}(\rho_1,...,\rho_n)}{\dim \C}=\dim X_1...\dim X_n\dim\Hom(\one, X_1\otimes...\otimes X_n).
$$
Note also that $Z_{\rho_j}=X_j\boxtimes X_j^*$, so $\dim Z_{\rho_j}=(\dim X_j)^2$. This implies the statement.  
\end{proof} 

In the original version of this article we asked the following question.

\begin{question} Is any spherical fusion category strongly Isaacs? 
Is it true at least for ribbon categories?\footnote{S. Palcoux has informed us that he checked many examples of commutative integral fusion rings which are perfect (i.e. have no invertible basis elements other than the unit) and pass the extended cyclotomic criterion (see \cite{LPR2}, 5.4), and all the Isaacs ones turned out to be strongly Isaacs.} 
\end{question}

As mentioned above, the first part of this question was answered negatively in \cite{BP}, a counterexample being the extended Haagerup category. However, the second part of the question (regarding ribbon categories) remains open.

\end{document}